\documentclass[12pt,intlimits]{article}

\usepackage{amssymb}
\usepackage{amsthm}
\usepackage{amsmath}

\usepackage{tikz}

\usepackage{dsfont}

\newcommand{\N}{\mathds{N}}

\newcommand{\R}{\mathds{R}}
\newcommand{\C}{\mathds{C}}

\newcommand{\1}{\mathds{1}}

\newcommand{\loc}{\mathrm{loc}}
\newcommand{\spt}{\mathrm{spt}}
\newcommand{\unif}{\mathrm{unif}}

\DeclareMathOperator{\TextIm}{Im}

\renewcommand{\Im}{\TextIm}

\newcommand{\M}{\mathcal{M}}
\renewcommand{\H}{\mathcal{H}}

\providecommand{\abs}[1]{\left\lvert#1\right\rvert}
\providecommand{\norm}[1]{\left\lVert#1\right\rVert}
\providecommand{\set}[1]{\left\{ #1\right\}}
\providecommand{\form}{\tau}

\newcommand{\from}{\colon}

\newcommand\llim{
\mathchoice{\vcenter{\hbox{${\scriptstyle{-}}$}}}
{\vcenter{\hbox{$\scriptstyle{-}$}}}
{\vcenter{\hbox{$\scriptscriptstyle{-}$}}}
{\vcenter{\hbox{$\scriptscriptstyle{-}$}}}}
\newcommand\rlim{
\mathchoice{\vcenter{\hbox{${\scriptstyle{+}}$}}}
{\vcenter{\hbox{$\scriptstyle{+}$}}}
{\vcenter{\hbox{$\scriptscriptstyle{+}$}}}
{\vcenter{\hbox{$\scriptscriptstyle{+}$}}}}

\makeatletter

\theoremstyle{plain} % default
\newtheorem{theorem}{Theorem}[section]

\newtheorem{lemma}[theorem]{Lemma}
\newtheorem{proposition}[theorem]{Proposition}
\theoremstyle{definition}

\newtheorem*{definition}{Definition}
\newtheorem{remark}[theorem]{Remark}

\makeatother

\usepackage{enumitem}

\begin{document}

\title{Absence of absolutely 
continuous spectrum for the Kirchhoff Laplacian on radial trees}

\author{Pavel Exner, Christian Seifert and Peter Stollmann}

\maketitle

\begin{abstract}
  In this paper we prove that the existence of absolutely continuous spectrum of the Kirchhoff Laplacian on a radial metric tree graph together with a finite complexity of the geometry of the tree implies that the tree is in fact eventually periodic.
  This complements the results by Breuer and Frank in \cite{BreuerFrank2009} in the discrete case as well as for sparse trees in the metric case.


\end{abstract}

\section{Introduction}

Quantum graphs and their discrete counterparts have been a subject of intense interest recently -- see \cite{BerkolaikoKuchment2013} for the bibliography -- both as a source of practically important models and an object of inspiring mathematical complexity. One of the important question concerns transport on such graphs, in particular, the presence or absence of the absolutely continuous spectrum of the corresponding Hamiltonian.

It was demonstrated in \cite{BreuerFrank2009} for the Laplacian $\Delta$ on a 
discrete radially symmetric rooted tree graph, that if the sequence of 
branching numbers $(b_n)$ is bounded and $\Delta$ has nonempty absolutely 
continuous spectrum then $(b_n)$ is eventually periodic, in other words, we may
 think of the geometry of the tree being eventually periodic. This result makes
 it possible to make a similar conclusion for metric graphs as long as they are
 equilateral using the known duality -- cf. \cite{Pankrashkin2012} and 
references therein. This tells us nothing, however, about the spectrum in the 
metric setting beyond the equilateral case when the geometry of the tree 
encoded in the edge lengths should come into play.

The aim of this note is address this question and to prove the analogous result for the Laplacian on radial
metric trees with Kirchhoff boundary conditions. The proof will combine three facts:
\begin{enumerate}
  \item
    a unitary equivalence of the Kirchhoff Laplacian on the tree with a direct sum of
self-adjoint operators on halflines -- see \cite[Theorem 3.5]{Solomyak2004},
  \item
    an adapted version of Remling's Oracle Theorem \cite[Theorem 2]{Remling2007} for such halfline operators,
  \item
    an adapted version of absence of absolutely continuous spectrum for ``finite local complexity'' \cite[Theorem 4.1]{KlassertLenzStollmann2011}.
\end{enumerate}

\noindent Sections 2 to 4 below are devoted respectively to these three facts; after discussing them we state and prove in Section 5 our main theorem and comment on its possible extensions.

\section{Radial tree graphs and unitary equivalence}

Let $\Gamma = (V,E)$ be a rooted radially symmetric metric tree graph with vertex set $V$ and edge set $E$.
Let $O\in V$ be the root, and for a vertex $v\in V$ of the $n$th generation of vertices let $b_n$ be the branching number of $v$,
i.e., the number of forward neighbors, and $t_n>0$ be the length of the path connecting the vertex $v$ with the root $O$.
We set $t_0:=0$ and $b_0:=1$.
In order to obtain a well-defined operator we will assume
\begin{equation}
\label{eq:edge}
  \inf_{n\in\N} (t_{n+1}-t_{n})>0,
\end{equation}
i.e., the the edge lengths are bounded away from $0$. Without loss of generality, we may also assume that
\begin{equation}
\label{eq:branching}
  \inf_{n\in\N} b_n>1,
\end{equation}
i.e., each vertex except the root will have at least two forward neighbors
(otherwise, by Kirchhoff boundary conditions, we can delete such vertices).

%   \begin{tikzpicture}[xscale=0.5,yscale=0.5]
%     \draw (0,0) node[left]{$O$};
%     \draw (0,0)--(1,0);
%       \draw (1,0)--(3,3);
% 	\draw (3,3)--(5,4.5);
% 	  \draw (5,4.5)--(7,5);
% 	  \draw (5,4.5)--(7,4);
% 	\draw (3,3)--(5.5,3);
% 	  \draw (5.5,3)--(7.5,3.5);
% 	  \draw (5.5,3)--(7.5,2.5);
% 	\draw (3,3)--(5,1.5);
% 	  \draw (5,1.5)--(7,2);
% 	  \draw (5,1.5)--(7,1);
%       \draw (1,0)--(3,-3);
% 	\draw (3,-3)--(5,-4.5);
% 	  \draw (5,-4.5)--(7,-5);
% 	  \draw (5,-4.5)--(7,-4);
% 	\draw (3,-3)--(5.5,-3);
% 	  \draw (5.5,-3)--(7.5,-3.5);
% 	  \draw (5.5,-3)--(7.5,-2.5);
% 	\draw (3,-3)--(5,-1.5);
% 	  \draw (5,-1.5)--(7,-2);
% 	  \draw (5,-1.5)--(7,-1);
%
%     \draw[dashed, gray] (3,-3) arc (-56.31:56.31:3.605);
%
%     \draw[dashed, gray] (5,1.5) arc (-36.87:36.87:2.5);
%     \draw[dashed, gray] (5,-4.5) arc (-36.87:36.87:2.5);
%
%     \draw[dashed, gray] (7,4) arc (-14.04:14.04:2.061);
%     \draw[dashed, gray] (7.5,2.5) arc (-14.04:14.04:2.061);
%     \draw[dashed, gray] (7,1) arc (-14.04:14.04:2.061);
%     \draw[dashed, gray] (7,-5) arc (-14.04:14.04:2.061);
%     \draw[dashed, gray] (7.5,-3.5) arc (-14.04:14.04:2.061);
%     \draw[dashed, gray] (7,-2) arc (-14.04:14.04:2.061);
%   \end{tikzpicture}

Let $H_\Gamma$ be the Laplacian on $\Gamma$ with Kirchhoff boundary conditions at the vertices except $O$, and Dirichlet boundary conditions at the root, i.e.,
$H_\Gamma$ is the unique self-adjoint operator associated with the form $\form$,
\begin{align*}
  D(\form) & := \set{u\in L_2(\Gamma);\; u'\in L_2(\Gamma),\, u(O) = 0,\, u\,\text{continuous on $\Gamma$}},\\
  \form(u,v) & := \int_{\Gamma} u(x) \overline{v(x)}\, dx.
\end{align*}

The following result was shown in \cite[Theorem 3.5]{Solomyak2004}, see also \cite[Proposition 5]{BreuerFrank2009}.
\begin{proposition}[{\cite[Theorem 3.5]{Solomyak2004}, \cite[Proposition 5]{BreuerFrank2009}}]
\label{prop:unitary_equivalence}
  $H_\Gamma$ is unitarily equivalent to
  \[A_0^{\rlim} \oplus \bigoplus_{k=1}^\infty (A_k^{\rlim}\otimes I_{\C^{b_1\cdots b_{k-1}(b_k-1)}}),\]
  where, for $k\geq 0$, $A_k^+$ is a linear operator in $L_2(t_k,\infty)$ defined by
  \begin{align*}
    D(A_k^{\rlim}) & := \Bigl\{u\in L_2(t_k,\infty);\; u \in W_2^2\Bigl(\bigcup_{n\geq k} (t_n,t_{n+1})\Bigr),\,u(t_k) = 0, \\
    & \qquad u(t_n\rlim) = \sqrt{b_n}u(t_n\llim),\, u'(t_n\rlim) = \frac{1}{\sqrt{b_n}} u'(t_n\llim) \quad(n> k)\Bigr\} \\
    (A_k^{\rlim} u)(t) & := -u''(t) \quad \Bigl(t\in \bigcup_{n\geq k} (t_n,t_{n+1})\Bigr).
  \end{align*}
\end{proposition}

The preceding proposition reduces the the study of (the spectrum of) $H_\Gamma$ to the study of (the spectra of) the operators $A_k^{\rlim}$.
These are operators on halflines.
We will describe such operators by means of measures in the following way.

\begin{definition}
 A measure $\mu$ on $\R$ is called atomic, if $\spt \mu$ is countable, i.e., if there exists $J\subseteq \N$ such that
$\mu = \sum_{n\in J} \beta_n \delta_{t_n}$ with suitable $(\beta_n),(t_n)$ in $\R$.
\end{definition}

For sequences $(b_n)$ in $(1,\infty)$, $(t_n)$ in $\R$ satisfying \eqref{eq:edge} and \eqref{eq:branching}, we associate a measure
$\mu = \sum_{n=1}^\infty \beta_n \delta_{t_n}$, where $\beta_n:=\frac{\sqrt{b_n}+1}{\sqrt{b_n}-1}$ ($n\in\N$).
Then define $H_\mu$ in $L_2(\R)$ by
\begin{align*}
  D(H_\mu) & := \Bigl\{u\in L_2(\R);\; u \in W_2^2\Bigl(\bigcup_{n\in\N_0} (t_n,t_{n+1})\Bigr),\\
& \qquad u(t_n\rlim) = \sqrt{b_n}u(t_n\llim),\, u'(t_n\rlim) = \frac{1}{\sqrt{b_n}} u'(t_n\llim) \quad(n\in\N)\Bigr\} \\
(H_\mu u)(t) & := -u''(t) \quad \Bigl(t\in \bigcup_{n\in\N_0} (t_n,t_{n+1})\Bigr).
\end{align*}
Then $H_\mu$ is self-adjoint and non-negative.
In case $(t_n)$ in $(0,\infty)$ we can also consider the corresponding halfline operators with Dirichlet boundary conditions at $0$, then denoted by $H_\mu^+$.
% \begin{align*}
% D(H_\mu^{\rlim}) & := \Bigl\{u\in L_2(0,\infty);\; u \in W_2^2\Bigl(\bigcup_{n\in\N_0} (t_n,t_{n+1})\Bigr),\,u(0) = 0,\\
% & \qquad u(t_n\rlim) = \sqrt{b_n}u(t_n\llim),\, u'(t_n\rlim) = \frac{1}{\sqrt{b_n}} u'(t_n\llim) \quad(n\in\N)\Bigr\} \\
% (H_\mu^{\rlim} u)(t) & := -u''(t) \quad \Bigl(t\in \bigcup_{n\in\N_0} (t_n,t_{n+1})\Bigr).
% \end{align*}
Note that each $A_k^{\rlim}$ in Proposition \ref{prop:unitary_equivalence} can be associated to some atomic measure $\mu$ such that $A_k^{\rlim}$ unitarily equivalent (by translation) to $H_\mu^{\rlim}$.

\section{The Oracle Theorem}

A signed Radon measure $\mu$ on $\R$ is said to be translation bounded, if
\[\norm{\mu}_\loc:= \sup_{x\in\R} \abs{\mu}([x,x+1]) < \infty.\]
Let $\M_{\loc,\unif}(\R)$ be the space of all translation bounded measures.

For an interval $I\subseteq \R$ and $C> 0$ let
\[\M^C(I):= \set{\1_{I}\mu;\; \mu\in \M_{\loc,\unif}(\R),\, \norm{\mu}_\loc\leq C}.\]
Equipped with the topology of vague convergence for measures $\M^C(I)$ is compact and hence metrizable (see, e.g., \cite[Proposition 4.1.2]{Seifert2012}).

For $\gamma>0$ let
\begin{align*}
\M_a^{\gamma}(\R)& :=\bigl\{\mu\in \M_{\loc,\unif}(\R);\; \mu\,\text{nonegative and atomic},\\
& \qquad \abs{s-t}\geq \gamma\quad(s,t\in\spt\mu, s\neq t)\bigr\}.
\end{align*}
Moreover, let
\[\M_a^{\gamma,\rlim}(\R):= \set{\mu\in\M_a^{\gamma}(\R);\; \spt\mu \subseteq [\gamma,\infty)}.\]
We will also need the subsets
\[\M_a^{C,\gamma}(\R) := \M^C(\R) \cap \M_a^{\gamma}(\R),\quad \M_a^{C,\gamma,\rlim}(\R) := \M^C(\R) \cap \M_a^{\gamma,\rlim}(\R).\]

\begin{remark}
  Let $\mu = \sum_{n\in\N} \beta_n\delta_{t_n} \in \M_a^{C,\gamma}(\R)$, where $\beta_n = \frac{\sqrt{b_n}+1}{\sqrt{b_n}-1}$.

  Let $s,t\in\R\setminus\spt\mu$, $s<t$. Let $u,v\in W_2^2((s,t)\setminus\spt\mu)$ satisfying
  \[u(t_n\rlim) = \sqrt{b_n}u(t_n\llim), \quad u'(t_n\rlim) = \frac{1}{\sqrt{b_n}} u'(t_n\llim) \quad(t_n\in(s,t)),\]
  and similarly for $v$.
  Then Green's formula holds in this case as well:
  \begin{align*}
      \int_s^t (-u'')(r) \overline{v}(r)\, dr - \int_s^t u(r) \overline{(-v'')}(r)\, dr
      & = W(u,\overline{v})(t) - W(u,\overline{v})(s),
  \end{align*}
  where $W(u,v)(t) = u'(t)v(t) - u(t)v'(t)$ is the Wronskian of $u$ and $v$ at $t$.
\end{remark}

\begin{remark}
\label{rem:Weyl-circle}
  Let $\mu\in \M_a^{C,\gamma}(\R)$, $z\in\C^+:=\set{z\in\C;\; \Im z>0}$.

  (a)
  Let $t\in\R\setminus\spt\mu$.
  Let $u_N(z,\cdot)$, $u_D(z,\cdot)$ be the two formal solutions of $H_\mu u = zu$ satisfying Neumann and Dirichlet conditions at $t$, i.e.
  \begin{align*}
    u_N(z,t) & = 1 & u_D(z,t) = 0 \\
    u_N'(z,t) & = 0 & u_D'(z,t) = 1.
  \end{align*}
  Let $b>t$.
  Consider the formal solution $u(z,\cdot)$ of $H_\mu u = zu$ satisfying a Robin condition with angle $\beta$ at $b$, i.e.
  \[\cos \beta u(z,b) + \sin\beta u'(z,b) = 0.\]
  We can write $u(z,\cdot) = u_N(z,\cdot)+m(z,t,b,\mu)u_D(z,\cdot)$
  Then
  \[m(z,t,b,\mu) = -\frac{u_N(z,b)\cot\beta + u_N'(z,b)}{u_D(z,b)\cot\beta + u_D'(z,b)}.\]
  Thus, $m(z,t,b,\mu)$ lies on the circle with center
  \[M(z,b,\mu) := \frac{W(u_N(z,\cdot),\overline{u_D(z,\cdot)})(b)}{W(u_D(z,\cdot),\overline{u_D(z,\cdot)})(b)}.\]
  and radius
  \[r(z,b,\mu):=\abs{\frac{W(u_N(z,\cdot),u_D(z,\cdot))(b)}{W(u_D(z,\cdot),\overline{u_D(z,\cdot)})(b)}} = \frac{1}{\abs{W(u_D(z,\cdot),\overline{u_D(z,\cdot)})(b)}}.\]
  By Green's formula we deduce that increasing the value $b$ yields smaller and smaller circles, where the smaller one is contained in the larger one.
  As in the Schr\"odinger case (see, e.g., \cite[Lemma III.1.4 and Corollary III.1.5]{CarmonaLacroix1990})
  one can show that $r(z,b,\mu)\to 0$ as $b\to\infty$, i.e., we are in the so-called limit point case.

  Taking $b\to\infty$, one therefore obtains $m(z,t,b,\mu)\to m_+(z,t,\mu)$ which lies in the interior of the circle to the value $b$, for all $b>0$.
  Analogous reasoning for $b\to-\infty$ yields a limit point $m_-(z,t,\mu)$.
  The limits are called the $m$-functions of $H_\mu$.

  (b)
  Similarly as in the case of Schr\"odinger operators one can show that there exist (unique up to multiplication by constants)
  formal solutions $u_\pm(z,\cdot)$ of $H_\mu u = zu$ lying in $L_2$ at $\pm\infty$.
  Then
  \[m_\pm(z,t,\mu) = \pm\frac{u_\pm'(z,t)}{u_\pm(z,t)} \quad(t\in\R\setminus\spt\mu).\]
\end{remark}

The $m$-functions are Herglotz functions, so its boundary values on the real line exist a.e. If $\mu\in\M_a^{C,\gamma,\rlim}(\R)$ then
$\C^+\ni z\mapsto m_+(z,0,\mu)$ depends only on the restriction of $\mu$ to $[0,\infty)$.
Note that the $m$-functions contain spectral information of the operator in the following way: the set
\[\Sigma_{ac}(H_\mu^{\rlim}) := \set{E\in\R;\; 0<\Im m_+(E+i0,0,\mu)<\infty}\]
is an essential support of the absolutely continuous spectrum of $H_\mu^+$. Hence, $H_\mu^+$ has absolutely continuous spectrum if and only if
$\Sigma_{ac}(H_\mu^{\rlim})$ has positive measure.

\begin{lemma}
\label{lem:conv_compacts}
  Let $C,\gamma>0$, $(\mu_n)$ in $\M_a^{C,\gamma,\rlim}(\R)$, $\mu\in\M_a^{C,\gamma}(\R)$, $\mu_n\to\mu$ vaguely. Then, for $t\in\R\setminus\spt\mu$, we have $m_\pm(\cdot,t,\mu_n)\to m_\pm(\cdot,t,\mu)$ uniformly on compact subsets of $\C^+$.
\end{lemma}

\begin{proof}
  Since $\mu_n\to \mu$, we conclude that for $s\in\R$ there exists $(s_n)$ such that $s_n\to s$ and $\mu_n(\set{s_n})\to \mu(\set{s})$. Let $b\notin\bigcup_{n\in\N} \spt\mu_n\cup \spt\mu$ and consider the
  formal solutions of $H_{\mu_n}u = zu$ satisfying Neumann and Dirichlet solutions at $t$. Then these solutions and their derivatives at $b$ converge locally uniformly in $z$ to the corresponding Neumann and Dirichlet solutions of $H_\mu u = zu$.
  Thus, also the circle center's and radii as is in Remark \ref{rem:Weyl-circle} converge locally uniformly in $z$.

  Now, let $\varepsilon>0$ and $K\subseteq \C^+$ be compact. There exists $b>0$ such that
  \[\sup_{z\in K} r(z,b,\mu)<\varepsilon.\]
  There exists $N\in\N$ such that for all $n\geq N$ we have
  \[\sup_{z\in K} \abs{M(z,b,\mu_n) - M(z,b,\mu)}<\varepsilon,\quad \sup_{z\in K} \abs{r(z,b,\mu_n) - r(z,b,\mu)}<\varepsilon.\]
  Thus,
  \[\sup_{z\in K} \abs{m_+(z,t,\mu_n)-m(z,t,\mu)} \leq 3\varepsilon \quad(n\geq N).\]
\end{proof}

\begin{remark}
  Let $\H:=\set{F\from\C^+\to\C^+;\; F\,\text{holomorphic}}$ be the set of Herglotz functions. Let $(F_n)$ be a sequence in $\H$, $F\in\H$.

  (a)
  We say that $F_n\to F$ in \emph{value distribution}, if
  \[\lim_{n\to\infty} \int_A \omega_{F_n(t)}(S)\,dt = \int_A \omega_{F(t)}(S)\,dt\]
  for all Borel sets $A,S\subseteq \R$ with $\lambda(A)<\infty$. Here, for $z=x+iy\in\C^+$ we have
  \[\omega_{z}(S):= \frac{1}{\pi} \int_S \frac{y}{(t-x)^2+y^2}\,dt\]
  and for $G\in\H$ and for a.a.~$t\in\R$ we have
  \[\omega_{G(t)}(S) :=\lim_{y\to0\rlim} \omega_{G(t+iy)}(S).\]
  Note that $G(t):=\lim_{y\to0\rlim} G(t+iy)$ exists for a.a.~$t\in\R$. For these $t$ we can define $\omega_{G(t)}(S)$ directly (which coincides with the definition via the limit).

  (b)
  As shown in \cite[Theorem  2.1]{Remling2011} the convergence $F_n\to F$ in value distribution holds if and only if $F_n\to F$ uniformly on compact subsets of $\C^+$.

  (c)
  Let $F,G\in\H$. Assume that $\omega_{F(t)}(S) = \omega_{G(t)}(S)$ for a.a.~$t\in A$ and for all $S\subseteq \R$. Then $F = G$ a.e.~on $A$, see \cite[paragraph before Theorem 2.1]{Remling2007}.
\end{remark}

\begin{definition}
  Let $C,\gamma>0$, $\mu\in\M_a^{C,\gamma}(\R)$, $\Lambda \subseteq \R$ measurable. Then $\mu$ is called \emph{reflectionless} on $\Lambda$, if
  \begin{equation}
  \label{eq:reflectionless}
  m_+(E+i0,t,\mu) = -\overline{m_-(E+i0,t,\mu)} \quad(\text{a.e.~$E\in\Lambda$})
  \end{equation}
  for some $t\in\R\setminus\spt\mu$.
  Let $\mathcal{R}^{C,\gamma}(\Lambda):= \set{\mu\in\M_a^{C,\gamma}(\R);\; \mu\,\text{reflectionless on}\, \Lambda}$
  be the set of reflectionless (atomic) measures on $\Lambda$ with parameters $C$ and $\gamma$.
\end{definition}

Note that if \eqref{eq:reflectionless} holds for some $t\in \R\setminus\spt\mu$ then it automatically holds for all $t\in \R\setminus\spt\mu$.

\begin{lemma}
\label{lem:refl_value}
  Let $\mu\in \M_a^{C,\gamma}(\R)$, $\mu(\set{0}) = 0$, $\Lambda\subseteq \R$ a Borel set. Then $\mu\in \mathcal{R}^{C,\gamma}(\Lambda)$ if and only if
  \begin{equation}
  \label{eq:refl_equiv}
  \int_B \omega_{m_-(E,0,\mu)}(-S)\,dE = \int_B \omega_{m_+(E,0,\mu))}(S)\, dE
  \end{equation}
  for all Borel sets $B\subseteq \Lambda$, $\lambda(B)<\infty$ and $S\subseteq \R$.
\end{lemma}

\begin{proof}
Assume that $\mu\in\mathcal{R}^{C,\gamma}(\Lambda)$. Then $\Im m_\pm(E,0,\mu)>0$ for a.a.~$E\in \Lambda$. Since $\mu$ is reflectionless on $\Lambda$, we have
\[m_+(E,0,\mu) = -\overline{m_-(E,0,\mu)}\]
for a.a.~$E\in \Lambda$. Since $\omega_{z}(-S) = \omega_{-\overline{z}}(S)$ for all $z\in\C^+$, we obtain
\[\int_B \omega_{m_-(E,0,\mu)}(-S)\,dE = \int_B \omega_{m_+(E,0,\mu)}(S)\, dE.\]

On the other hand, assume \eqref{eq:refl_equiv}. Lebesgue's differentiating theorem yields
\[\omega_{m_-(E,0,\mu)(-S)} = \omega_{m_+(E,0,\mu)}(S) \quad(\text{a.e.~$E\in\Lambda$}).\]
Since $\omega_{z}(-S) = \omega_{-\overline{z}}(S)$ for all $z\in\C^+$ we obtain
\[\omega_{-\overline{m_-(E,0,\mu)}}(S) = \omega_{m_+(E,0,\mu)}(S) \quad(\text{a.e.~$E\in \Lambda$}).\]
Thus, $m_+(E,0,\mu) = -\overline{m_-(E,0,\mu)}$ for a.a.~$E\in \Lambda$.
\end{proof}

Next, we again consider measures on halflines:
\begin{align*}
  \M_+^C:=\set{\1_{(0,\infty)}\mu;\; \mu\in\M^C(\R)},\\
  \M_-^C:=\set{\1_{(-\infty,0)}\mu;\; \mu\in\M^C(\R)}.
\end{align*}
Note that $\M_+^C$ can be identified with $\M^C((0,\infty))$, and we use the topology and metric $d_+$ from this space.
Similar identifications apply to $\M_-^C$.
Then the restriction maps $\M^C(\R)\to \M_\pm^C$ are continuous and thus $(\M_\pm^C,d_\pm)$ are compact. Furthermore, these restriction maps are injective.
Let
\begin{align*}
  \mathcal{R}_+^{C,\gamma}(\Lambda) & := \set{\1_{(0,\infty)}\mu;\; \mu\in \mathcal{R}^{C,\gamma}(\Lambda)}\subseteq \M_+^C,\\
  \mathcal{R}_-^{C,\gamma}(\Lambda) & := \set{\1_{(-\infty,0)}\mu;\; \mu\in \mathcal{R}^{C,\gamma}(\Lambda)}\subseteq \M_-^C,
\end{align*}
equipped with the metrics $d_\pm$.

We can now prove the analogon of \cite[Proposition 2]{Remling2007} (which was proven for the Schr\"odinger case) in our setting.

\begin{proposition}
\label{prop:compact}
  Let $\Lambda\subseteq \R$ be measurable and $C,\gamma>0$. Then
  $(\mathcal{R}^{C,\gamma}(\Lambda),d)$ and $(\mathcal{R}_\pm^{C,\gamma}(\Lambda),d_\pm)$ are compact and
  the restriction maps $\mathcal{R}^{C,\gamma}(\Lambda)\to\mathcal{R}_\pm^{C,\gamma}(\Lambda)$ are homeomorphisms.
\end{proposition}

\begin{proof}
  It suffices to show that $\mathcal{R}^{C,\gamma}(\Lambda)$ is closed. Let $(\mu_n)$ in $\mathcal{R}^{C,\gamma}(\Lambda)$, $\mu\in\M^C(\R)$ an atomic measure, $\mu_n\to \mu$.
  Then Lemma \ref{lem:conv_compacts} yields $m_\pm(\cdot,0,\mu_n)\to m_\pm(\cdot,0,\mu)$ uniformly on compact subsets of $\C^+$.
  Thus, also $m_\pm(\cdot,0,\mu_n)\to m_\pm(\cdot,0,\mu)$ in value distribution. By Lemma \ref{lem:refl_value}, for $n\in\N$ we have
  \[\int_B \omega_{m_-(E,0,\mu_n)}(-S)\,dE = \int_B \omega_{m_+(E,0,\mu_n)}(S)\, dE\]
  for all Borel sets $B\subseteq \Lambda$, $\lambda(B)<\infty$ and $S\subseteq \R$.
  Taking the limit $n\to\infty$ yields
  \[\int_B \omega_{m_-(E,0,\mu)}(-S)\,dE = \int_B \omega_{m_+(E,0,\mu)}(S)\, dE\]
  for all Borel sets $B\subseteq \Lambda$, $\lambda(B)<\infty$ and $S\subseteq \R$.
  Again applying Lemma \ref{lem:refl_value} we obtain $\mu\in \mathcal{R}^{C,\gamma}(\Lambda)$.

  Since the restriction maps are continuous, $(\mathcal{R}_\pm^{C,\gamma}(\Lambda),d_\pm)$ are compact as continuous images of a compact space.
  Since the restriction maps are bijective and continuous between these compact metric spaces, their inverses are continuous as well.
\end{proof}

For $\mu\in\M^{C}(\R)$ we write $S_x\mu:=\mu(\cdot+x)$ for the translate by $x$. Then we can define the $\omega$ limit set of $\mu$ as
\[\omega(\mu):=\set{\nu\in\M^{C}(\R);\; \text{there exists $(x_n)$ in $\R$}: x_n\to \infty,\, d(S_{x_n}\mu,\nu)\to 0}.\]
Note that if $\mu\in \M_a^{C,\gamma}(\R)$, then $\omega(\mu) \subseteq\M_a^{C,\gamma}(\R)$.

The following result was proven in \cite[Theorem 16]{BreuerFrank2009}.

\begin{proposition}[{\cite[Theorem 16]{BreuerFrank2009}}]
\label{prop:omega_refl}
  Let $C,\gamma>0$, $\mu\in\M_a^{C,\gamma,+}(\R)$. Then $\omega(\mu)\subseteq \mathcal{R}^{C,\gamma}(\Sigma_{ac}(H_\mu^+))$.
\end{proposition}

We can now prove a version of Remling's Oracle Theorem (\cite[Theorem 2]{Remling2007}) for our setting.

\begin{theorem}
\label{thm:oracle_thm}
 Let $\Lambda\subseteq \R$ be a Borel set of positive Lebesgue measure, $\varepsilon>0$, $a,b,\in\R$, $a<b$, $C>0$. Then there exist $L>0$ and a continuous function
\[\triangle\from \M^C(-L,0) \to \M^C(a,b)\]
such that $\triangle$ maps atomic measures to atomic measures
so that the following holds. If there is $\gamma>0$ and $\mu\in \M_a^{C,\gamma,\rlim}(\R)$ such that $\Sigma_{ac}(H_\mu^{\rlim})\supseteq \Lambda$, then there exists $x_0>0$ so that for all $x\geq x_0$ we have
\[d(\triangle(\1_{(-L,0)}S_{x}\mu), \1_{(a,b)}S_{x}\mu) <\varepsilon.\]
\end{theorem}

\begin{proof}
  We follow the proof of \cite[Theorem 2]{Remling2007}, but replace the application of Theorem 3 in there (in step 4) by Proposition \ref{prop:omega_refl}.

  (i)
  By compactness it suffices to prove the statement for \emph{some} metric $d$ that generates the topology of vague convergence.

  Let $J_-:=(-L,0)$ and $J_+:=(a,b)$. Let $d_\pm$ be the metric on $M^C(J_\pm)$.
  For $\mu\in\M^C(\R)$ let $\mu_\pm$ be the restriction of $\mu$ to $J_\pm$.

  Fix a metric $d$ such that $d$ dominates $d_\pm$: if $\mu,\nu\in \M^C(\R)$ then
  \[d_-(\mu_-,\nu_-)\leq d(\mu,\nu),\quad d_+(\mu_+,\nu_+)\leq d(\mu,\nu).\]

  (ii)
  By Proposition \ref{prop:compact} the mapping $\mathcal{R}_-^{C,\gamma}(\Lambda)\to \mathcal{R}_+^{C,\gamma}(\Lambda)$ is uniformly continuous.
  Hence, by the definition of the topologies
  we may find $L>0$ and $0<\delta<\varepsilon<1$ such that if $\nu,\tilde{\nu}\in \mathcal{R}^{C,\gamma}(\Lambda)$, then
  \[d_-(\nu_-,\tilde{\nu}_-)< 5\delta \Longrightarrow d_+(\nu_+,\tilde{\nu}_+) < \varepsilon^2.\]

  (iii)
  The set
  \[\mathcal{R}_{J_-}^{C,\gamma}(\Lambda) := \set{\mu_-;\; \mu\in \mathcal{R}^{C,\gamma}(\Lambda)}\]
  is compact by Proposition \ref{prop:compact}.
  Since $\M^{C,\gamma}(J_-)$ is compact, the closed $\delta$-neighborhood
  \[\overline{U}_\delta = \set{\mu_-\in \M^{C,\gamma}(J_-);\; \exists\,\nu\in\mathcal{R}^{C,\gamma}(\Lambda): d_-(\mu_-,\nu_-)\leq \delta}\]
  is compact, too.
  Hence, there exist $\mathcal{F}\subseteq \mathcal{R}^{C,\gamma}(\Lambda)$ finite so that the balls of radius $2\delta$ around $\mathcal{F}$ cover $\overline{U}_\delta$.
  Define $\triangle(\nu_{-}):= \nu_{+}$ for $\nu\in\mathcal{F}$.

  (iv)
  For $\sigma\in \overline{U}_\delta$ define
  \[\triangle(\sigma):=\frac{\sum_{\nu\in\mathcal{F}} (3\delta - d_-(\sigma,\nu_-))^+\triangle (\nu_-)}{\sum_{\nu\in\mathcal{F}} (3\delta - d_-(\sigma,\nu_-))^+}.\]
  Then $\triangle(\sigma)$ is atomic for $\sigma\in\overline{U}_\delta$ and $\triangle\from \overline{U}_\delta\to \M^{C}(J_+)$ is continuous.
  Moreover, for all $\tilde{\nu}\in\mathcal{F}$ with $d_-(\sigma,\tilde{\nu}_-)<2\delta$ we have $d_-(\nu_-,\tilde{\nu}_-)<5\delta$ for all $\nu\in\mathcal{F}$ contributing to the sum.
  Thus, (ii) implies $d_+(\nu_+,\tilde{\nu}_+)<\varepsilon^2$ for these $\nu$ and by \cite[Lemma 2]{Remling2007} we obtain
  \[d_+(\triangle(\sigma),\tilde{\nu}_+)<\varepsilon\]
  for sufficiently small $\varepsilon$. Note that for every $\sigma\in \overline{U}_\delta$ there exists $\tilde{\nu}\in\mathcal{F}$ such that $d_-(\sigma,\tilde{\nu}_-)<2\delta$.

  So, $\triangle$ is defined on $\overline{U}_\delta$ and maps every $\sigma\in \overline{U}_\delta$ to some nonegative atomic measure $\triangle(\sigma)$. The extension theorem of Dugundji and Borsuk \cite[Chapter II, Theorem 3.1]{BessagaPelczynski1975} yields a continuous extension of $\triangle$ to $\M^C(J_-)$.
  Since this extension is obtained by convex combinations, $\triangle$ maps atomic measures to atomic measures.

  (v)
  Now, choose $\mu\in\M_a^{C,\gamma,+}(\R)$ with $\Sigma_{ac}(H_\mu^+)\supseteq \Lambda$. Then there exists $x_0>0$ such that
  \[d(S_{x}\mu,\omega(\mu))<\delta \quad(x\geq x_0),\]
  i.e., for fixed $x\geq x_0$ there exists $\nu\in\omega(\mu)$ such that $d(S_{x}\mu,\nu)<\delta$. By Proposition \ref{prop:omega_refl} we have
  $\nu\in \mathcal{R}^{C,\gamma}(\Lambda)$. We thus obtain
  \[d_\pm((S_{x}\mu)_\pm,\nu_\pm)<\delta.\]
  Hence, $(S_{x}\mu)_-\in \overline{U}_\delta$, so there exists $\tilde{\nu}\in\mathcal{F}$ such that
  \[d_\pm((S_{x}\mu)_-,\tilde{\nu}_-)<2\delta.\]
  By (iv) we obtain
  \[d_+(\triangle((S_{x}\mu)_-),\tilde{\nu}_+)<\varepsilon.\]
  Since also $d_-(\nu_-,\tilde{\nu}_-)<3\delta$, (ii) implies $d_+(\nu_+,\tilde{\nu}_+)<\varepsilon^2$. Therefore, we finally conclude
  \[d_-(\triangle((S_{x}\mu)_-),(S_{x}\mu)_+) < \delta + \varepsilon + \varepsilon^2 < 3\varepsilon.\]
\end{proof}

\section{Finite local complexity}

Let us recall some definitions from \cite{KlassertLenzStollmann2011}.

\begin{definition}
  A \emph{piece} is a pair $(\nu,I)$ consisting of a left-closed right-open interval $I\subseteq\R$ with positive length
  $\lambda(I) >0$ (which is then called the \emph{length} of the piece) and a $\nu\in\M_{\loc,\unif}(\R)$ supported on $I$.
  We abbreviate pieces by $\nu^I$.
  A \emph{finite piece} is a piece of finite length.
  We say $\nu^I$ \emph{occurs} in a measure $\mu$ at $x\in\R$, if $\1_{x+I}\mu$
  is a translate of $\nu$.

	The \emph{concatenation} $\nu^I=\nu_1^{I_1}\mid \nu_2^{I_2}\mid \ldots$ of a finite or countable family
    $(\nu_j^{I_j})_{j\in N}$, with $N\subseteq \N$, of finite pieces is defined by
	\begin{align*}
		I & = \left[\min I_1,\min I_1 + \sum_{j\in N} \lambda(I_j)\right),\\
		\nu & = \nu_1+\sum_{j\in N,\,j\geq 2} \nu_j\Big(\cdot-\Big(\min I_1 + \sum_{k=1}^{j-1} \lambda(I_k) - \min I_j\Big)\Big).
	\end{align*}
	We also say that $\nu^I$ is \emph{decomposed} by $(\nu_j^{I_j})_{j\in N}$.
\end{definition}

\begin{definition}
	Let $\mu$ be a measure on $\R$. We say that $\mu$ has the \emph{finite decomposition property} (f.d.p.),
	if there exist a finite set $\mathcal{P}$ of finite pieces (called the \emph{local pieces})
	and $x_0\in\R$, such that $\1_{[x_0,\infty)}\mu^{[x_0,\infty)}$ is a translate of a
	concatenation $v_1^{I_1}\mid \nu_2^{I_2}\mid\ldots$ with $\nu_j^{I_j}\in\mathcal{P}$ for all $j\in\N$.
	Without restriction, we may assume that $\min I = 0$ for all $\nu^I\in \mathcal{P}$.

	A measure $\mu$ has the \emph{simple finite decomposition property} (s.f.d.p.), if it has the f.d.p.~with a decomposition such that there is $\ell>0$ with the following property: Assume that the two pieces
	\[\nu_{-m}^{I_{-m}} \mid \ldots \mid \nu_{0}^{I_{0}} \mid \nu_{1}^{I_{1}} \mid \ldots \mid \nu_{m_1}^{I_{m_1}} \quad \text{and} \quad
	 \nu_{-m}^{I_{-m}} \mid \ldots \mid \nu_{0}^{I_{0}} \mid \mu_{1}^{J_{1}} \mid \ldots \mid \mu_{m_2}^{J_{m_2}}\]
	occur in the decomposition of $\mu$ with a common first part $\nu_{-m}^{I_{-m}} \mid \ldots \mid \nu_{0}^{I_{0}}$ of length at least $\ell$ and such that
	\[\1_{[0,\ell)}(\nu_{1}^{I_{1}} \mid \ldots \mid \nu_{m_1}^{I_{m_1}}) = \1_{[0,\ell)}(\mu_{1}^{J_{1}} \mid \ldots \mid \mu_{m_2}^{J_{m_2}}),\]
	where $\nu_j^{I_j}$, $\mu_k^{J_k}$ are pieces from the decomposition (in particular, all belong to $\mathcal{P}$ and start at $0$) and the latter two concatenations are of lengths at least $\ell$. Then
	\[\nu_1^{I_1} = \mu_1^{J_1}.\]
\end{definition}

\begin{lemma}
\label{lem:properties_mu}
  Let $(t_n)$ in $\R$, $(b_n)$ in $(1,\infty)$, $\mu= \sum_{n=1}^\infty \beta_n \delta_{t_n}$ be an atomic measure, where $\beta_n = \frac{\sqrt{b_n}+1}{\sqrt{b_n}-1}$ ($n\in\N$). Then:
  \begin{enumerate}
    \item
      If $(t_n)$ and $(b_n)$ satisfy \eqref{eq:edge} and \eqref{eq:branching} then $\mu \in \M_{\loc,\unif}(\R)$.
    \item
      $\mu$ is eventually periodic if and only if $((t_{n+1}-t_n,b_n))$ is eventually periodic.
    \item
      $\mu$ has the s.f.d.p. if $\set{t_{n+1}-t_{n};\; n\in\N}$ and $\set{b_n;\; n\in\N}$ are finite.
  \end{enumerate}
\end{lemma}

\begin{proof}
  In order to prove (a) we remark that $(\beta_n)$ is bounded by \eqref{eq:branching} and that
  \[\norm{\mu}_\loc \leq \frac{\sup_{n\in\N} \beta_n}{\inf_{n\in\N} (t_{n+1}-t_{n})}.\]
  Part (b) is clear by definition of $\mu$. To prove (c) note that $\mu$ can be
  decomposed by
  $\mathcal{P}:=\set{(\1_{[t_{n},t_{n+1})}\beta_n\delta_{t_n})(\cdot+t_n);\;
  n\in\N}$ and this set is, by assumption, finite.
  To show that the decomposition is simple note that
  \[\1_{[0,\ell)}(\nu_{1}^{I_{1}} \mid \ldots \mid \nu_{m_1}^{I_{m_1}}) = \1_{[0,\ell)}(\mu_{1}^{J_{1}} \mid \ldots \mid \mu_{m_2}^{J_{m_2}}),\]
  where all the $\nu_j$'s and $\mu_j$'s are elements from $\mathcal{P}$ directly implies that $\nu_1^{I_1} = \mu_1^{J_1}$.
\end{proof}

\begin{theorem}
\label{thm:nonempty_eventually_periodic}
  Let $(t_n)$ in $(0,\infty)$ and $(b_n)$ satisfy \eqref{eq:edge} and \eqref{eq:branching}. Assume that the corresponding measure $\mu$ has the s.f.d.p.
  Then, if $H_\mu^+$ has nonempty absolutely continuous spectrum the measure $\mu$ is eventually periodic.
\end{theorem}

\begin{proof}
  Assume that $\sigma_{ac}(H_\mu^+)$ is nonempty. Then $\Sigma_{ac}(H_\mu^+)$ has positive measure.
  Let $\ell$ and $\mathcal{L}$ be the minimum and maximum of the length of the finitely many pieces of $\mu$ according to s.f.d.p., respectively, and $G:=\set{x_k;\; k\in\N_0}$ be the grid.
  Let
  \[\mathcal{D}:= \set{\1_{(-1,\ell)}(S_{x}\mu);\; x\in G}.\]
  Then $\mathcal{D}$ is finite and hence there exists $\varepsilon>0$ such that $d(\nu_1,\nu_2)>2\varepsilon$ for $\nu_1,\nu_2\in \mathcal{D}$ with $\nu_1\neq \nu_2$.
  Let $a:=-1$ and $b:=\ell$ and choose $L>\ell$ according to Theorem \ref{thm:oracle_thm}.

  We construct a coarser grid $G_L\subseteq G$ by $y_0:=x_0$, $y_{k+1}-y_k\in [L,L+\mathcal{L}]$. Since $G\cap [y_k,y_{k+1}]$ is finite, also
  \[\mathcal{P}_L:= \set{\bigl(S_{y_{k}}(\1_{[y_k,y_{k+1})}\mu), S_{y_{k}}([y_k,y_{k+1}))\bigr);\; k\in\N}\]
  is finite. Hence, there exists $k\in\N$ such that infinitely many translates of $\1_{[y_k,y_{k+1})}\mu$ occur in $\mu$, so that the corresponding parts of $G$ are translates of $G\cap[y_k,y_{k+1})$.

  Let $z^1:=y_{m+1}$, where $y_m$ is one of the corresponding points in $G_L$, $m>k$ and denote $y^1:=y_{k+1}$. By Theorem \ref{thm:oracle_thm} applied with $x=y^1$ and $x=z^1$ we obtain
  \[d(\1_{(-1,\ell)}(S_{y^1}\mu), \1_{(-1,\ell)}(S_{z^1}\mu)) \leq 2\varepsilon.\]
  Hence, $\1_{(-1,\ell)}(S_{y^1}\mu) = \1_{(-1,\ell)}(S_{z^1}\mu)$ by the choice of $\varepsilon$.
  Moreover, we know that the pieces of $\mu$ starting at $y_k$ and $y_m$, respectively, are decomposed in the same way.
  The s.f.d.p~yields that the pieces starting at $y^1$ and $z^1$ are translates from each other,
  i.e., for $y^2:=\min G\cap (y^1,\infty)$ and $z^2:=\min G\cap (z^1,\infty)$ we have
  \[z^2-z^1 = y^2-y^1,\quad \text{and}\quad S_{y^1}(\1_{[y^1,y^2)}\mu) = S_{z^1}(\1_{[z^1,z^2)}\mu).\]
  Thus, $\1_{[y_k,y^2)}\mu$ is a translate of $\1_{[y_m,z^2)}\mu$. Iterating, we obtain sequences $(y^n)$ and $(z^n)$ in $G$ such that $\1_{[y_k,y^n)}\mu$ is a translate of $\1_{[y_m,z^n)}\mu$ for all $n\in\N$.
  Since $y^{n+1}-y^n = z^{n+1}-z^n\geq \ell$ for all $n\in\N$ we obtain that $\1_{[y_k,\infty)}\mu$ is a translate of $\1_{[y_m,\infty)}\mu$, i.e., $\mu$ is eventually periodic.
\end{proof}

\section{Absence of absolutely continuous spectrum on trees}

We can now state our main theorem, which is the analogue of \cite[Theorem
1]{BreuerFrank2009} for radially symmetric metric tree graphs.

\begin{theorem} \label{thm:main}
  Let $(t_n)$ in $(0,\infty)$ and $(b_n)$ in $(1,\infty)$ satisfy \eqref{eq:edge} and \eqref{eq:branching}. Assume that $\set{t_{n+1}-t_{n};\; n\in\N}$ and $\set{b_n;\; n\in\N}$ are finite.
  Then, if $H_\Gamma$ has nonempty absolutely continuous spectrum the sequence $((t_{n+1}-t_{n},b_n))$ is eventually periodic.
\end{theorem}

\begin{proof}
  By Proposition \ref{prop:unitary_equivalence} it suffices to prove the statement for all $A_k^+$, $k\geq 0$.

  Note that, for $k\geq 0$, $A_k^+$ is unitarily equivalent to $H_{\mu_k}^+$, where $\mu_k:= S_{t_k}(\1_{(t_k,\infty)}\mu)$ and $\mu$ is associated to $((t_n,b_n))$.
  Indeed, the unitary transformation is just the shift by $t_k$.
  Since $\mu$ has the s.f.d.p.~by Lemma \ref{lem:properties_mu}, clearly also $\mu_k$ has the s.f.d.p.~for all $k\geq 0$.
  Theorem \ref{thm:nonempty_eventually_periodic} proves that $\mu_k$ is eventually periodic for all $k\geq 0$. Thus, also $\mu$ is eventually periodic. By Lemma \ref{lem:properties_mu} we conclude that
  $((t_{n+1}-t_{n},b_n))$ is eventually periodic.
\end{proof}

\begin{remark}
Kirchhoff conditions are the simplest but by far not the only way to couple the tree edges in a self-adjoint way. In particular, a wide class of boundary conditions on rooted radially symmetric metric tree graphs was considered in \cite{ExnerLipovsky2010} 
and it was shown, in analogy with the corresponding result in 
\cite{BreuerFrank2009}, that if such a tree is sparse the absolutely continuous
 spectrum is absent. One can treat such boundary conditions also in the present
 context. With the help of \cite[Theorem 6.4]{ExnerLipovsky2010} one can prove 
an Oracle Theorem similar to Theorem \ref{thm:oracle_thm} for these more 
general boundary conditions, and furthermore, one can  associate with them two 
or three atomic measures \cite[Section VI]{ExnerLipovsky2010}; if all of them 
are s.f.d.p. and at least one is nontrivial, then one can derive the result on 
absence of absolutely continuous spectrum analogous to Theorem~\ref{thm:main}. 
We stress the nontriviality requirement: there is a subset of boundary 
conditions \cite[Example~7.1]{ExnerLipovsky2010} which are mapped by the 
unitary equivalence of Section~2 to free motion on the family of halflines, 
hence the absolutely continuous spectrum is present in such cases, in fact 
covering the whole positive halfline, irrespective of the edge lengths.
\end{remark}

% \section{More general boundary conditions}
%
% \textbf{The statements in this section still have to be done.}
%
% In \cite{ExnerLipovsky2010} more general boundary conditions on rooted radially symmetric metric tree graph were considered. Nevertheless, absence of absolutely continuous spectrum
% can also be derived for these boundary conditions.
%
% With the help of \cite[Theorem 6.4]{ExnerLipovsky2010} one can prove an Oracle Theorem as in Theorem \ref{thm:oracle_thm} for more general boundary conditions.
% Also, to each boundary condition  in question we may associate either two or
% three measures atomic measures, cf. \cite[Section VI]{ExnerLipovsky2010}.
%
% If all of them have the s.f.d.p., then we obtain the corresponding result on absence of absolutely continuous spectrum.

\subsection*{Acknowledgments}
The research was supported by the Czech Science Foundation  within
the project P203/11/0701.

\bigskip

\noindent
Pavel Exner \\
Doppler Institute for Mathematical Physics and Applied Mathematics \\
Faculty of Nuclear Sciences and Physical Engineering \\
Czech Technical University \\
B\v{r}ehov{\'a} 7 \\
11519 Prague, Czech Republic \\
and \\
Nuclear Physics Institute ASCR \\
25068 \v{R}e\v{z} near Prague, Czech Republic \\
{\tt exner@ujf.cas.cz}

\bigskip

\noindent
Christian Seifert \\
Institut f\"ur Mathematik \\
Technische Universit\"at Hamburg-Harburg \\
21073 Hamburg, Germany \\
{\tt christian.seifert@tuhh.de}

\bigskip

\noindent
Peter Stollmann \\
Technische Universit\"at Chemnitz \\
Fakult\"at f\"ur Mathematik \\
09107 Chemnitz, Germany \\
{\tt P.Stollmann@mathematik.tu-chemnitz.de}

\begin{thebibliography}{99}
\bibitem{BerkolaikoKuchment2013} G.\;Berkolaiko and P.\;Kuchment,
\emph{Introduction to Quantum Graphs}, Amer. Math. Soc., Providence, R.I., 2013.

\bibitem{BessagaPelczynski1975} C.\;Bessaga and A.\;Pelczynski,
\emph{Selected topics in infinite-dimensional topology},
Mathematical Monographs, vol. 58. Polish Scientific, Warsaw (1975).

\bibitem{BreuerFrank2009} J.\;Breuer and R.\;Frank,
\emph{Singular spectrum for radial trees}.
Rev.\ Math.\ Phys. {\bf 21}(7), 929--945 (2009).

\bibitem{CarmonaLacroix1990} R.\;Carmona and J.\;Lacroix,
\emph{Spectral Theory of Random Schr\"odinger Operators}.
Birkh\"auser Boston, 1990.

\bibitem{ExnerLipovsky2010} P.\;Exner and J.\;Lipovsk\'{y},
\emph{On the absence of absolutely continuous spectra for Schr\"odinger operators on radial tree graphs}.
J.\ Math.\ Phys.\ {\bf 51}, 122107 (2010).

\bibitem{KlassertLenzStollmann2011} S.\;Klassert, D.\;Lenz and P.\;Stollmann,
\emph{Delone measures of finite local complexity and applications to spectral theory of one-dimensional continuum models of quasicrystals}.
Discrete Contin.\ Dyn.\ Syst.\ {\bf 29}(4), 1553--1571 (2011).

\bibitem{NaimarkSolomyak2004} K.\;Naimark and M.\;Solomyak,
\emph{Eigenvalue estimates for the weighted Laplacian on metric trees}.
Proc.\ London\ Math.\ Soc.\ {\bf 80}(3), 690--724 (2000).

\bibitem{Pankrashkin2012}
K.\;Pankrashkin, \emph{Unitary dimension reduction for a class of self-adjoint extensions with
applications to graph-like structures}, J.\ Math.\ Anal.\ Appl. {\bf 396}, 
640--655 (2012).

\bibitem{Remling2007} C.\;Remling,
\emph{The absolutely continuous spectrum of one-dimensional Schr\"o\-dinger 
operators}.
Math Phys Anal Geom {\bf 10}, 359--373 (2007).

\bibitem{Remling2011} C.\;Remling,
\emph{The absolutely continuous spectrum of Jacobi matrices}.
Annals of Math.\ {\bf 174}, 125--171 (2011).

\bibitem{Seifert2012} C.\;Seifert,
\emph{Measure-perturbed one-dimensional Schr\"odinger operators -- A continuum model for quasicrystals}.
Dissertation thesis, Chemnitz University of Technology (2012). url: \\ \texttt{http://nbn-resolving.de/urn:nbn:de:bsz:ch1-qucosa-102766}

\bibitem{SobolevSolomyak2002} A.V.\;Sobolev and M.\;Solomyak,
\emph{Schr\"odinger operators on homogeneous metric trees: spectrum in gaps}.
Rev.\ Math.\ Phys.\ {\bf 14}, 421--468 (2002).

\bibitem{Solomyak2004} M.\;Solomyak,
\emph{On the spectrum of the Laplacian on regular metric trees}.
Waves\ Random\ Media {\bf 14}, 155--171 (2004).
\end{thebibliography}
\end{document}